\documentclass[reqno,11pt]{amsart}
\usepackage[english]{babel}
\usepackage{amssymb,verbatim,microtype,graphicx}
\usepackage[T1,T5]{fontenc}
\usepackage[utf8]{inputenc} 

\newtheorem{theorem}{Theorem}[section]

\newtheorem{lemma}[theorem]{Lemma}

\newtheorem{corollary}[theorem]{Corollary}
\newtheorem{conjecture}[theorem]{Conjecture}

\theoremstyle{definition}
\newtheorem{definition}[theorem]{Definition}

\newtheorem{remark}[theorem]{Remark}

\newtheorem*{ackn}{Acknowledgements}

\numberwithin{equation}{section}

\newcommand{\C}{\mbox{$\mathbb{C}$}}

\newcommand{\B}{\mbox{$\mathbb{B}$}}

\newcommand{\Capo}{\mbox{Cap}_{\Omega}}

\newcommand{\Ker}{\mbox{$\mathcal{N}$}}
\newcommand{\E}{\mbox{$\mathcal{E}$}}
\newcommand{\Eo}{\mbox{$\mathcal{E}_{0}$}}

\newcommand{\F}{\mbox{$\mathcal{F}$}}

\newcommand{\PSH}[1]{\mbox{$\mathcal{PSH}(#1)$}}
\newcommand{\PSHM}[1]{\mbox{$\mathcal{PSH}^-(#1)$}}

 \usepackage{hyperref}
\hypersetup{
    unicode=false,
    pdftoolbar=true,
    pdfmenubar=true,
    pdffitwindow=false,
    pdfstartview={FitH},
    pdftitle={Geodesic connectivity and rooftop envelopes in the Cegrell classes},
    pdfauthor={Ahag, Czyz, Lu and Rashkovskii},
    colorlinks=true,
   linkcolor=black,
    citecolor=black,
    filecolor=black,
    urlcolor=black}

\begin{document}

\title[Geodesics and envelopes in the Cegrell classes]{Geodesic connectivity and rooftop envelopes\\ in the Cegrell classes}

\author{Per \AA hag}\address{Department of Mathematics and Mathematical Statistics\\ Ume\aa \ University\\SE-901~87 Ume\aa \\ Sweden}\email{per.ahag@umu.se}

\author{Rafa\l\ Czy{\.z}}\address{Faculty of Mathematics and Computer Science, Jagiellonian University, \L ojasiewicza~6, 30-348 Krak\'ow, Poland}
\email{rafal.czyz@im.uj.edu.pl}

\author{Chinh H. Lu}\address{Laboratoire Angevin de Recherche en Math\'{e}matiques (LAREMA), Universit\'{e} d'Angers (UA), 2 Boulevard de Lavoisier, 49000 Angers, France}
\email{hoangchinh.lu@univ-angers.fr}

\author{Alexander Rashkovskii}\address{Department of Mathematics and Physics, University of Stavanger, 4036 Stavanger, Norway}\email{alexander.rashkovskii@uis.no}

\date{\today}

\begin{abstract} This study examines geodesics and plurisubharmonic envelopes within the Cegrell classes on bounded hyperconvex domains in $\mathbb{C}^n$. We establish that solutions possessing comparable singularities to the complex Monge-Ampère equation are identical, affirmatively addressing a longstanding open question raised by Cegrell. This achievement furnishes the most general form of the Bedford-Taylor comparison principle within the Cegrell classes. Building on this foundational result, we explore plurisubharmonic geodesics, broadening the criteria for geodesic connectivity among plurisubharmonic functions with connectable boundary values. Our investigation also delves into the notion of rooftop envelopes, revealing that the rooftop equality condition and the idempotency conjecture are valid under substantially weaker conditions than previously established, a finding made possible by our proven uniqueness result. The paper concludes by discussing the core open problems within the Cegrell classes related to the complex Monge-Ampère equation.
\end{abstract}

\maketitle

\tableofcontents

\section{Introduction}
Since the appearance of the celebrated works of Bedford-Taylor \cite{BT76}, Calabi \cite{Cal57}, and Yau \cite{Yau78}, complex Monge-Ampère equations have occupied a central place in complex analysis and geometry. Acting on a smooth plurisubharmonic function, the Monge-Ampère operator is the determinant of its Hessian matrix, thereby defining a positive volume form. This is a consequence of the fact that the associated $(1,1)$-form is positive, and the determinant of the matrix is realized as the wedge product of this $(1,1)$-form. For non-smooth plurisubharmonic functions, it remains positive but in the weak sense of currents, making the definition of the Monge-Ampère operator for these functions a delicate task, as the wedge product of positive currents cannot be straightforwardly taken.

The breakthrough by Bedford-Taylor \cite{BT76} succeeded in defining the complex Monge-Ampère operator for locally bounded plurisubharmonic functions, thus paving the way for pluripotential theory -- the multidimensional counterpart of potential theory in the complex plane. An indispensable tool in this theory is the comparison principle, which in its simplest form states that if two bounded plurisubharmonic functions within a bounded domain share the same boundary values and Monge-Ampère measure, they are identical.

An example by Shiffman and Taylor \cite{siu} demonstrates that the Monge-Ampère operator cannot be defined for all unbounded plurisubharmonic functions. In his seminal works \cite{Ceg98,Ceg04,Ceg08}, Cegrell introduced several classes of unbounded plurisubharmonic functions for which the Monge-Ampère operator is well-defined and exhibits all the expected continuity properties. Subsequently, the Cegrell classes have attracted extensive study by numerous authors. Noteworthy is \cite[Theorem 4.14]{ACCP09}, which established that the Monge-Ampère equation admits a solution provided a subsolution exists, thus extending the renowned bounded subsolution theorem by Ko{\l}odziej \cite{Kol96}. Since the mid-1970s, the mathematical community has faced the intricate challenge of establishing a comparison principle for Monge-Ampère measures of plurisubharmonic functions that charge pluripolar sets, a challenge that has constrained progress considerably. It is worth mentioning that, according to \cite[Corollary 5.4]{Dem93} and \cite[Example 3.4]{Zer97}, solutions can possess non-comparable singularities. A significant advancement was achieved in \cite[Theorem 3.6]{ACCP09}, where it was shown that comparable solutions are equivalent, provided their Monge-Ampère measure is integrable against some negative plurisubharmonic function. However, as illustrated in \cite{Ceg08}, this condition is restrictive; for example, a bounded plurisubharmonic function with zero boundary values may have a Monge-Ampère measure that is not integrable against any negative plurisubharmonic function. The necessity of this condition has been a core open question in pluripotential theory.

In the first main result of this paper, we eliminate the above integrability condition, thereby establishing an analog of the Bedford-Taylor comparison principle that applies to all functions within the Cegrell classes.

\bigskip

\noindent {\bf Theorem~\ref{thm: uniqueness in N}}. \emph{
    Let $\Omega$ be a bounded hyperconvex domain in $\mathbb{C}^n$, and let $u,v$ be plurisubharmonic functions in the Cegrell class $\Ker(H)$, for some $H\in \E(\Omega)$. If
\[
u\preceq  v\quad \text{and} \quad (dd^c u)^n \leq (dd^c v)^n,
\]
then $u\geq v$.
}

\medskip

The result was previously established for cases where the measure $(dd^c v)^n$ does not assign mass to pluripolar sets, as demonstrated in \cite[Corollary 3.2]{ACCP09}. The notation $u \preceq v$ signifies that $u$ is more singular than $v$, with further details provided in Definition~\ref{def: singular}. Here, $\E(\Omega)$ is the largest class of plurisubharmonic functions on which the Monge-Ampère operator is well-defined and exhibits continuity along decreasing sequences. The condition $u, v \in \Ker(H)$ implies that $u$ and $v$ share boundary values, as specified by $H$. Definitions and fundamental properties of these classes are elaborated upon in the following section.

Our proof is based on a technique of plurisubharmonic rooftop envelopes, a method recently utilized in multiple papers, including \cite{DDL5},  \cite{GL21a, GL22, GL23Crelle}, and \cite{LN22,Sal23}. It is also noteworthy that \cite{Sal23} established uniqueness for solutions involving non-pluripolar measures and non-Kähler forms. We establish here some important properties of the envelopes and there Monge-Ampère measures, see Theorem~\ref{thm: envelope} and its corollaries, which has allowed us to treat the general case of functions in Cegrell classes with controlled boundary behavior.

As applications of Theorem~\ref{thm: uniqueness in N}, we provide affirmative answers to several questions concerning geodesic connectivity and rooftop envelopes. The concept of plurisubharmonic geodesics was first introduced in Mabuchi's seminal work on constant scalar curvature Kähler metrics \cite{Mab86}. Subsequently, Semmes \cite{Sem92} and Donaldson \cite{Don99} independently showed that Mabuchi geodesics can be understood as solutions to certain degenerate homogeneous complex Monge-Ampère equations. This concept has been further refined to describe geodesics as the upper envelopes of subgeodesics, a perspective that Berman-Berndtsson \cite{BB22}, Abja \cite{Abj19}, Abja-Dinew \cite{AD21}, and Rashkovskii \cite{Ras17} adapted to the local setting. We direct readers to \cite{Ras23} for a comprehensive overview of this topic.

Geodesics can be seen as optimal plurisubharmonic interpolations between pairs of plurisubharmonic functions, somewhat similar to the classical Calder\'on complex interpolation in Banach spaces. While constructing a geodesic segment $u_t$, $0<t<1$, between two bounded plurisubharmonic functions $u_0$ and $u_1$ is straightforward, the general case poses challenges in boundary behavior of the geodesics at the endpoints. Namely, the limits of $u_t$ as $t$ approaches $0$ and $1$ may deviate from $u_0$ and $u_1$, leading to a disconnection. In \cite{Rash22}, it was shown that strong singularities within the domain and at its boundary remain invariant under plurisubharmonic interpolation, preventing connectivity when singularities in the data functions differ. Key to addressing this was the asymptotic rooftop envelopes $P[u](v)$ and the Green-Poisson residual function $g_u = P[u](0)$. For definitions, see Section~\ref{sec:RIC}. Notably, $g_u$ is defined by the asymptotic behavior of $u$ near its negative infinity points, both inside the domain and on the boundary.

As shown in \cite{Rash22}, Darvas' method \cite{Dar17} can be readily adapted to prove that $u_0$ and $u_1$ are geodesically connected if and only if $P[u_0](u_1) = u_1$ and $P[u_1](u_0) = u_0$.  Furthermore, it was conjectured that these conditions are equivalent to having $g_{u_0} = g_{u_1}$, signifying identical leading terms in their singularities. This has been confirmed in~\cite{Ras23} for certain classes of plurisubharmonic  functions. Leveraging Theorem~\ref{thm: uniqueness in N}, we significantly extend this connectivity criterion:

\bigskip

\noindent {\bf Theorem~\ref{thm: geodesic connectivity}}.\emph{
  Given $H_0,H_1 \in \E$ that are connectable by a plurisubharmonic geodesic, and $u_0\in\Ker(H_0)$, $u_1\in\Ker(H_1)$, then $u_0$ and $u_1$ can be connected by a  plurisubharmonic geodesic segment if and only if
\begin{equation}
    \label{eq: geo connectivity intro}
    u_0\leq g_{u_1}\quad \text{ and }\quad u_1\leq g_{u_0}.
\end{equation}
   In particular, if $g_{H_0}=g_{H_1}$, then \eqref{eq: geo connectivity intro} is equivalent to $g_{u_0}=g_{u_1}$.}

\medskip

 Apart from the uniqueness Theorem~\ref{thm: uniqueness in N}, the second crucial step of the proof is the rooftop equality $P[u](v)= P(g_u,v)$, conjectured in \cite{Rash22} in the general case and proved there and in \cite{Ras23} for particular cases. In  Theorem~\ref{thm: RE conj N}, we have extended it significantly.

\bigskip

\noindent {\bf Theorem~\ref{thm: RE conj N}.}\emph{
    Assume $H_1\in \E$, $H_2\in \PSHM{\Omega}$, and
    \[
P[H_1](H_2)= P(g_{H_1},H_2).
\]
Then $P[u](v)= P(g_u,v)$,  for all $u\in \Ker(H_1)$, $v\in \Ker(H_2)$. }

\medskip

In other words, our finding above confirms that if the rooftop equality is verified for $H_1,H_2$, then it equally applies to every $u \in \Ker(H_1)$, $v\in \Ker(H_2)$. Since the rooftop equality trivially holds for $H_1=0$, by Theorem~\ref{thm: RE conj N} we see that it holds for all $u\in \Ker$, $v\in \PSHM{\Omega}$.

We also highlight that the functions $H_2$ and $v$ may not be in $\mathcal{E}$. Analyzing the Monge-Ampère operator presents challenges due to its undefined status for both $P[u](v)$ and $P(g_u, v)$.  Nonetheless, the singular part of their difference is manageable, and by applying the plurisubharmonic envelope, one secures a function in $\mathcal{E}$, where the Monge-Ampère measure vanishes.

In the global context on a compact Kähler manifold $(X,\omega)$ of dimension $n$, for given $\omega$-plurisubharmonic functions $u,v$ in the same relative full mass class $\mathcal{E}(X,\omega,\phi)$, the rooftop equality
$$P[u](v)=P(P[u],v)$$
is satisfied if the non-pluripolar Monge-Ampère measure $(\omega + dd^c u)^n$ possesses positive mass \cite[Theorem~3.14]{DDL23}. Due to the monotonicity of non-pluripolar Monge-Ampère masses, this positive mass condition for $u$ ensures the same for its asymptotic rooftop $P[u](0)=\phi$. In contrast, in our local setting, the non-pluripolar Monge-Ampère measure for $g_u$ equals zero, and the monotonicity does not apply. Moreover, unlike in the global context, functions may have infinite non-pluripolar Monge-Ampère mass. Additionally, the boundary behavior of plurisubharmonic functions plays an immense role, highlighting an undeveloped area in pluripotential theory.

As emphasized, the boundary functions $H$ play a crucial role in our results. Ph\d{a}m \cite{PhamPC} has announced that functions in $\E$, whose Monge-Ampère measure integrates a non-trivial negative plurisubharmonic function, belong to $\Ker(H)$ when $H$ is maximal. We refer to $H$ as the boundary values for these functions. In Theorem~\ref{bvinN}, we provide a proof that slightly diverges from Ph\d{a}m's unpublished approach.

\bigskip

\noindent {\bf Theorem~\ref{bvinN}.} \emph{If $u\in \E$ and $\int_{\Omega} (-w)(dd^c u)^n<+\infty$, for some $w\in \mathcal{PSH}^-(\Omega)$, $w<0$, then $u$ has boundary values.}

\medskip

The paper is organized as follows. In Section~\ref{Sec:Preliminaries}, we provide the necessary definitions and preliminary results, including fundamental properties of envelopes and rooftop envelopes. In Section~\ref{sect: uniqueness}, we prove Theorem~\ref{thm: uniqueness in N}, establishing the general form of the Bedford-Taylor comparison principle. Building on this foundational result, Section~\ref{sec:RIC} is devoted to exploring rooftop envelopes, idempotency, and geodesic connectivity, where we prove Theorem~\ref{thm: geodesic connectivity}, and Theorem~\ref{thm: RE conj N}, among other results. Section~\ref{Sec: BV} addresses the boundary values of plurisubharmonic functions. Finally, Section~\ref{sect: Open Problems} discusses the core open problems related to the complex Monge-Ampère equation, aiming to foster a deeper understanding and stimulate further advancements in the field.

\begin{ackn}
The project PARAPLUI ANR-20-CE40-0019 and the Centre Henri Lebesgue ANR-11-LABX-0020-01 partially support this work. The project started when the fourth-named author was visiting Jagiellonian University and continued during his stay at Université d'Angers. He is grateful to both institutions for their support.
\end{ackn}

\section{Preliminaries}\label{Sec:Preliminaries}

In this section, we introduce necessary definitions and establish foundational facts essential for the remainder of this paper. For additional details and a broader context, readers are encouraged to consult sources such as \cite{ACCP09,Ceg98,Ceg04,Ceg08,Cz09,GZbook}.

Throughout this paper, we assume that $\Omega \Subset \mathbb{C}^n$ is a bounded \emph{hyperconvex} domain, i.e., there exists $\psi \in \PSHM{\Omega}$ such that $\{z \in \Omega : \psi(z) < c \} \Subset \Omega$ for all $c < 0$. For technical results related to which pseudoconvex domains are hyperconvex, see~\cite{BennyLisaHakan}. A fundamental component of this paper involves the Cegrell classes, defined as follows:
\begin{align*}
\Eo(\Omega) &= \left\{\varphi \in \mathcal{PSH^-} \cap L^{\infty}(\Omega) : \lim_{z \rightarrow \xi} \varphi(z) = 0, \ \forall \xi \in \partial \Omega, \int_{\Omega} (dd^c \varphi)^n < +\infty \right\}, \\
\F(\Omega) &= \left\{\varphi \in \mathcal{PSH^-}(\Omega) : \exists \ [\varphi_j], \varphi_j \in \Eo(\Omega), \varphi_j \searrow \varphi, \sup_j \int_{\Omega} (dd^c \varphi_j)^n < +\infty \right\}, \\
\E(\Omega) &= \left\{\varphi \in \mathcal{PSH^-}(\Omega) : \exists \varphi_U \in \F(\Omega) \text{ such that } \varphi_U = \varphi \text{ on } U, \forall \ U \Subset \Omega \right\}.
\end{align*}
Let $[\Omega_j]$ denote a \emph{fundamental sequence} in the sense that it is an increasing sequence of strictly pseudoconvex subsets of the bounded hyperconvex domain $\Omega$, such that $\Omega_j \Subset \Omega_{j+1}$ for every $j \in \mathbb{N}$, and $\bigcup_{j=1}^{\infty} \Omega_j = \Omega$. Then, if $u \in \E$, and $[\Omega_j]$ is a fundamental sequence, we define
\[
u^j = \sup \left\{ \varphi \in \PSHM{\Omega} : \varphi \leq u \text{ on } \bar{\Omega}_j \right\},
\]
and
\[
\tilde{u} = \left(\lim_{j \to +\infty} u^j\right)^*.
\]
The function $\tilde{u}$ is the \emph{smallest maximal plurisubharmonic majorant} of $u$. Alternatively, it can be described as follows. Since the rooftop (see
Definition~\ref{def:rooftop}) of two maximal plurisubharmonic functions is evidently maximal, the family of all maximal plurisubharmonic majorants has a unique minimal element, and it is precisely $\tilde u$.

Set
\[
\Ker = \left\{u \in \E : \tilde{u} = 0 \right\}.
\]
We say that functions in $\Ker$ have (generalized) \emph{boundary values} of $0$. A less widely recognized characterization of the class $\Ker$, as found in \cite{Lisa}, states:  For a function $u \in \E$, the following assertions are equivalent:
\begin{enumerate}
    \item $u \in \Ker$,
    \item there is a plurisubharmonic function $\varphi = \sum_{j=1}^{\infty} \varphi_j$, $\varphi_j \in \F$, such that $u\geq \varphi$ on $\Omega$.
\end{enumerate}

 To summarize some additional facts about these classes:

\medskip

\begin{itemize}\itemsep2mm
\item[$i)$] $\E$ is the largest class of negative plurisubharmonic functions where the complex Monge-Ampère operator is well-defined;

\item[$ii)$] $\log|z_2| \notin \E(\mathbb{B})$, $\log|z| \in \F(\mathbb{B})$ with $(dd^c \log|z|)^n = (2\pi)^n\delta_0$, ($\mathbb B$ is the unit ball in $\mathbb C^n$);

\item[$iii)$] $\Eo \subsetneq \F \subsetneq \Ker \subsetneq \E$;

\item[$iv)$] $\Eo, \F, \Ker, \E$ are convex cones;

\item[$v)$] For any $u\in\Ker$ we have
\[
\limsup_{\Omega\ni z \rightarrow \xi} u(z) = 0 \quad \text{ for all } \xi\in\partial\Omega,
\]
but there are functions in $\Ker$ that additionally satisfy:
\[
\liminf_{\Omega\ni z \rightarrow \xi} u(z) = -\infty \quad \text{ for all } \xi\in\partial\Omega;
\]
\item[$vi)$] $\F=\{u\in \Ker: \int_{\Omega} (dd^cu)^n<+\infty\}$.
\end{itemize}

\medskip

Next, we introduce the Cegrell classes with generalized boundary values. It is important to emphasize that in Definition~\ref{Prel: CC}, we assume $H \in \mathcal{PSH}^-(\Omega)$, rather than the more common assumption of $H \in \E$.

\begin{definition}\label{Prel: CC}
Let $\mathcal{K} \in \{\Eo, \F, \Ker, \E\}$. A plurisubharmonic function $u$ on $\Omega$ belongs to the class $\mathcal{K}(\Omega, H) (= \mathcal{K}(H))$, $H \in \PSHM{\Omega}$, if there exists a function $\varphi \in \mathcal{K}$ such that
\[
H \geq u \geq \varphi + H.
\]
\end{definition}
For any subset $\mathcal{K} \subseteq \mathcal{E}(\Omega)$, we introduce the notation
\[
\mathcal{K}^a = \{\varphi \in \mathcal{K} : (dd^c \varphi)^n \text{ vanishes on all pluripolar sets in } \Omega\}.
\]

We shall next discuss two aspects of the decomposition of Monge-Ampère measures and see that they coincide. First, let us recall the Cegrell-Lebesgue decomposition theorem: If $\mu$ is a non-negative Radon measure, then it can be decomposed into a regular (non-pluripolar)  and singular (pluripolar) part
\[
\mu=\mu_r^C+\mu_s^C,
\]
in such a way that
\[
\mu_r^C=f(dd^c\varphi)^n,
\]
where $\varphi\in \Eo$ and $f\geq 0$, $f\in L^1_{loc}((dd^c\varphi)^n)$, and $\mu_s^C$ is carried by a pluripolar subset of $\Omega$. Furthermore, if $\mu=(dd^cu)^n$, $u\in \E$, then in addition we know that $\mu_s^C$ is carried by $\{z\in \Omega: u(z)=-\infty\}$.

On the second decomposition, for $u \in \PSH{\Omega}$, and with the notation $u_t=\max(u,-t)$, for $t>s>0$ we obtain
\begin{flalign}
{\bf 1}_{\{u>-s\}} (dd^c u_s)^n &= {\bf 1}_{\{u_t>-s\}}   (dd^c \max(u_t,-s))^n\nonumber \\
& =  {\bf 1}_{\{u_t>-s\}}   (dd^c u_t)^n = {\bf 1}_{\{u>-s\}}   (dd^c u_t)^n \nonumber \\
& \leq {\bf 1}_{\{u>-t\}}   (dd^c u_t)^n, \label{eq: NP MA}
\end{flalign}
where in the first and third lines we have used
\[
\{u_t>-s\} = \{u>-s\} \subset \{u>-t\}, \; u_s = \max(u_t,-s),
\]
and in the second line, we have utilized the Bedford-Taylor maximum principle. An important consequence of \eqref{eq: NP MA} is that
the non-negative Radon measures ${\bf 1}_{\{u>-t\}}(dd^c u_t)^n$ increase with $t$, and hence we can define
\[
\mu_r(u)  = \lim_{t\to +\infty} {\bf 1}_{\{ u>-t\}}(dd^c u_t)^n.
\]
If $\mu_r$ is locally finite, then it defines a positive Radon measure in $\Omega$. From the second line of \eqref{eq: NP MA} we also have
\[
{\bf 1}_{\{u>-s\}} (dd^c u_s)^n  = {\bf 1}_{\{u>-s\}}   (dd^c u_t)^n = {\bf 1}_{\{u>-s\}} {\bf 1}_{\{u>-t\}}   (dd^c u_t)^n.
\]
Then, by letting $t\to +\infty$, we arrive at
\begin{equation}
	\label{eq: NP plurifine property}
	{\bf 1}_{\{u>-s\}} (dd^c u_s)^n  = {\bf 1}_{\{u>-s\}}  \mu_r(u), \; \forall s>0.
\end{equation}
It follows from \cite[Theorem 2.1]{BGZ09} that, for $u\in \E$, the non-pluripolar Monge-Ampère measure of $u$ coincides with its full Monge-Ampère measure outside the pluripolar locus:
\[
\mu_r(u)  ={\bf 1}_{\{u>-\infty\}} (dd^c u)^n,
\]
and it puts no mass on pluripolar sets. Thus, setting $\mu_s(u)= {\bf 1}_{\{u=-\infty\}} (dd^c u)^n$, we have
\[
(dd^c u)^n =\mu_s(u) + \mu_r(u)=\mu_s^C(u)+ \mu_r^C(u).
\]
The letter $r$ (respectively, $s$) stands for the regular (respectively, singular) part of the Monge-Ampère measure $(dd^c u)^n$.

We will need the following maximum principle on several occasions. The full mass version of Theorem~\ref{thm: maximum principle} can be found in~\cite{NP09}.

\begin{theorem}\label{thm: maximum principle}
	If $u, v \in \E$, then
	\[
	 \mu_r(\max(u,v))    \geq {\bf 1}_{\{u\geq v\}} \mu_r(u)  +  {\bf 1}_{\{u < v\}} \mu_r(v).
	\]
	If, in addition, $u\leq v$ then
	\[
	 {\bf 1}_{\{u = v\}} \mu_r(v)  \geq {\bf 1}_{\{u = v\}} \mu_r(u).
	\]
\end{theorem}
\begin{proof}
	With the standard notation $u_t= \max(u,-t)$, $v_t=\max(v,-t)$, $t>0$, we get
	\[
	(dd^c \max(u_t,v_t))^n   \geq {\bf 1}_{\{u_t\geq v_t\}}   (dd^c u_t)^n  +  {\bf 1}_{\{u_t< v_t\}}  (dd^c v_t)^n.
	\]
Multiplying with ${\bf 1}_{\{\min(u,v)>-s\}}$, $s<t$, and then using~\eqref{eq: NP plurifine property}, we obtain
	\begin{flalign*}
			{\bf 1}_{\{\min(u,v)>-s\}} \mu_r(\max(u,v))   & \geq {\bf 1}_{\{\min(u,v)>-s\}}  {\bf 1}_{\{u\geq v\}}   \mu_r(u)   \\
			&+ {\bf 1}_{\{\min(u,v)>-s\}}  {\bf 1}_{\{u< v\}}  \mu_r(v).
	\end{flalign*}
Letting $s\to +\infty$, we arrive at the conclusion.
\end{proof}

A central tool we shall use is an envelope construction.
\begin{definition}
    For a function $h : \Omega \rightarrow \mathbb{R} \cup \{-\infty\}$, which is bounded from above, we define the envelope $P(h)$ as the upper semicontinuous regularization of the function
\[
x \mapsto \sup \{ u(x) : u \in \PSH{\Omega}, u \leq h \text{ quasi-everywhere in } \Omega \},
\]
with the convention that $\sup\emptyset = -\infty$. If no plurisubharmonic function lies below $h$ quasi-everywhere, then we simply define $P(h)$ to be identically $-\infty$.
\end{definition}
Here, quasi-everywhere means outside a pluripolar set.
If there exists $u \in \PSH{\Omega}$ such that $u \leq h$ quasi-everywhere, then $P(h) \in \PSH{\Omega}$, and by Choquet's lemma, there exists an increasing sequence $[u_j] \subset \PSH{\Omega}$ such that $u_j \leq h$ quasi-everywhere in $\Omega$ and $(\lim_j u_j)^* = P(h)$. The set
\[
\{x \in \Omega : \lim_j u_j(x) < P(h)(x)\}
\]
is pluripolar. Since a countable union of pluripolar sets is pluripolar, we infer that $P(h) \leq h$ quasi-everywhere in $\Omega$.

We say that $h$ is quasi-continuous if, for each $\varepsilon > 0$, there exists an open set $U$ such that $\Capo(U, \Omega) < \varepsilon$, and the restriction of $h$ on $\Omega \setminus U$ is continuous. By analogy, a set $E$ is quasi-open if, for each $\varepsilon > 0$, there exists an open set $U$ such that
\[
\Capo(U \setminus E \cup E \setminus U) \leq \varepsilon.
\]
Recall that the Monge-Ampère capacity is defined as
\[
\Capo(E) = \sup \left\{ \int_E (dd^c u)^n : u \in \PSH{\Omega}, -1 \leq u \leq 0 \right\}.
\]
Following Xing~\cite{Xing96}, we say that a sequence $[u_j]$  \emph{converges in capacity} to $u$ if, for every $\varepsilon > 0$ and any compact set $K \Subset \Omega$,
\[
\lim_{j \to +\infty} \Capo(\{ |u_j - u| > \varepsilon \} \cap K) = 0.
\]
It follows that monotone convergence implies convergence in capacity. We will use the following fact.
\begin{lemma}\label{l2.4}
If $-C \leq u_j \leq 0$ are plurisubharmonic functions converging in capacity to $u \in \PSH{\Omega}$, and $E$ is a quasi-open set, then
\[
\liminf_{j \to +\infty} \int_E (dd^c u_j)^n \geq \int_E (dd^c u)^n.
\]
\end{lemma}
\begin{proof}
We note that $(dd^c u_j)^n$ weakly converges to $(dd^c u)^n$. Since $E$ is quasi-open and the measures $(dd^c u_j)^n$ are uniformly dominated by $\Capo$, the result follows.
\end{proof}

\begin{theorem}\label{thm: lsc of NP}
    Assume $[u_j]$ is a sequence in $\E$ that converges in capacity to $u \in \E$. Then
   \[
    \liminf_{j \to +\infty} \mu_r(u_j) \geq \mu_r(u).
 \]
\end{theorem}
\begin{proof}
  Fix a smooth test function $\chi$ in $\Omega$. Fix $C > 0, \varepsilon > 0$ and consider
    \[
f_j^{C,\varepsilon} := \frac{\max(u_j + C, 0)}{\max(u_j + C, 0) + \varepsilon}, \ j \in \mathbb{N},
    \]
    and
    \[
u_{j}^{C} := \max(u_j, -C).
    \]
    Observe that for $C$ fixed, the functions $u_{j}^{C} \geq -C$ are uniformly bounded in $\Omega$ and converge in capacity to $u^{C} = \max(u, -C)$ as $j \to \infty$. Moreover, $f_{j}^{C,\varepsilon} = 0$ if $u_j \leq -C$. By the locality of the non-pluripolar product, we can write
    \[
f_j^{C,\varepsilon}  \chi (dd^c u_j^C)^n  = f_j^{C,\varepsilon}  \chi (dd^c u_j)^n.
    \]
For each $C, \varepsilon$ fixed, the functions $f_j^{C,\varepsilon}$ are quasi-continuous, uniformly bounded (with values in $[0,1]$), and converge in capacity to $f^{C,\varepsilon}$, where $f^{C,\varepsilon}$ is defined by
    \[
f^{C,\varepsilon} := \frac{\max(u + C, 0)}{\max(u + C, 0) + \varepsilon}.
    \]
With the information above, we can apply Xing's theorem \cite{Xing00} to get that
    \[
f_j^{C,\varepsilon}  \chi (dd^c u_j^C)^n  \rightarrow f^{C,\varepsilon}  \chi (dd^c u^C)^n \ \text{as}\ j \to \infty,
    \]
in the weak sense of measures in $\Omega$. In particular, since $0 \leq f^{C,\varepsilon} \leq 1$, we have that
    \begin{eqnarray*}
        \liminf_{j \to +\infty} \int_{\{u_j > -C\}} \chi \mu_r(u_j) &= &\liminf_{j \to +\infty} \int_{\{u_j > -C\}} \chi (dd^c u_j^C)^n\\ & \geq &  \liminf_{j \to \infty} \int_{\Omega} f_j^{C,\varepsilon}  \chi (dd^c u_j^C)^n\\
        & = & \int_{\Omega} f^{C,\varepsilon}  \chi (dd^c u^C)^n.
    \end{eqnarray*}
    Now, letting $\varepsilon \to 0$ and then $C \to +\infty$, by the definition of the non-pluripolar measure, we obtain
    \[
        \liminf_{j \to +\infty} \int_{\Omega} \chi \mu_r(u_j)
        \geq  \int_{\Omega}  \chi \mu_r(u).
    \]
\end{proof}
\begin{corollary}
    If $[u_j]$ is a sequence in $\E$ that increases almost everywhere to $u \in \PSHM{\Omega}$, then the following weak convergence holds:
    \[
    \mu_r(u_j) \to \mu_r(u), \; \mu_s(u_j) \to \mu_s(u).
    \]
\end{corollary}
\begin{proof}
    We have $\mu_s(u_j) \geq \mu_s(u)$, for all $j$, and $(dd^c u_j)^n \to (dd^c u)^n$ by \cite{Ceg18}. The result thus follows from Theorem~\ref{thm: lsc of NP}.
\end{proof}

\begin{theorem}\label{thm: envelope}
	If $h$ is quasi-continuous in $\Omega$ and $P(h) \not \equiv -\infty$, then
	\[
	{\bf 1}_{\{P(h) < h\}}  \mu_r (P(h))   = 0.
	\]
\end{theorem}
\begin{proof}
	Arguing as in the proof of Theorem~2.7 in \cite{DDL23}, we can find a sequence $[h_j]$ of lower semi-continuous functions in $\Omega$ such that $h_j \searrow h$ quasi-everywhere in $\Omega$. By the standard balayage method, for all $j$, we have
	\[
	\int_{\{P(h_j) < h_j\}} (dd^c P(h_j))^n =0.
	\]
	For $k<j$, using $\{P(h_k)<h\} \subset \{P(h_j)<h_j\}$, we then have
	\[
	\int_{\{P(h_k) < h\}} (dd^c P(h_j))^n =0.
	\]
	Fixing $t>0$, and using~\eqref{eq: NP plurifine property} for $\max(P(h_j),-t)$, we obtain
	\[
	\int_{\{P(h_k) < h, P(h)>-t\}} (dd^c \max(P(h_j),-t))^n =0.
	\]
The set $\{P(h_k) < h, P(h)>-t\}$ is quasi-open and the sequence $(\max(P(h_j),-t))_j$ is uniformly bounded, thus letting $j\to +\infty$, we obtain, by Lemma~\ref{l2.4},
	\[
	\int_{\{P(h_k) < h, P(h)>-t\}} (dd^c \max(P(h),-t))^n =0.
	\]
We finally let $k\to +\infty$, and then $t\to +\infty$ to arrive at the result.
\end{proof}

Having established the foundational aspects of plurisubharmonic functions and the Cegrell classes, we now focus on rooftop envelopes, a significant tool in contemporary pluripotential theory. We proceed to define this indispensable concept and present a proof of Corollary~\ref{cor: NP of rooftop}, which will play a significant role in Section~\ref{sec:RIC}.

\begin{definition}\label{def:rooftop}
The \emph{rooftop envelope} $P(u,v)$, for any two plurisubharmonic functions $u$ and $v$, is defined as the plurisubharmonic envelope of $\min(u,v)$, the largest plurisubharmonic function lying below $\min(u,v)$.
\end{definition}

\begin{corollary}\label{cor: NP of rooftop}
If $u,v\in \E$ then
    \[
    \mu_r(P(u,v))\leq {\bf 1}_{\{P(u,v)=u\}}\mu_r(u) + {\bf 1}_{\{P(u,v)=v, P(u,v)<u\}} \mu_r(v).
    \]
In particular, if $\mu$ is a positive measure such that $\mu_r(u)\leq \mu$ and $\mu_r(v)\leq \mu$, then $\mu_r(P(u,v))\leq \mu$.
\end{corollary}
\begin{proof}
    We first note that $P(u,v)\not \equiv -\infty$ because $u+v\leq P(u,v)$. By Theorem~\ref{thm: envelope}, $\mu_r(P(u,v))$ is supported on the contact set $D=D_1\cup D_2$, where
    \[
    D_1=\{P(u,v)=u\}\; \text{and}\; D_2=\{P(u,v)=v\}\cap \{P(u,v)<u\}.
    \]
    Theorem \ref{thm: maximum principle} then yields
    \[
    {\bf 1}_{D_1}\mu_r(P(u,v))\leq {\bf 1}_{D_1} \mu_r(u), \;  {\bf 1}_{D_2}\mu_r(P(u,v))\leq {\bf 1}_{D_2} \mu_r(v),
    \]
which finish the proof.
\end{proof}

\begin{corollary}\label{cor: NP of P(u-v,0)}
    For all $u,v\in \PSHM{\Omega}$, we have
    \[
    \mu_r(P(u-v,0)) \leq \mu_r(u).
    \]
\end{corollary}
\begin{proof}
    Since $P(u-v,0)+v\leq u$ with equality on the contact set $D=\{P(u-v,0)=u-v\}$, we have, by the maximum principle, Theorem \ref{thm: maximum principle},
    \[
    {\bf 1}_D \mu_r(P(u-v,0)) + {\bf 1}_D\mu_r(v) \leq {\bf 1}_D \mu_r(u).
    \]
    By Theorem \ref{thm: envelope},  $\mu_r(P(u-v,0))$ is supported on $D$, and this finishes the proof.
\end{proof}

\section{Uniqueness in the Cegrell Classes} \label{sect: uniqueness}

Building on the foundational concepts introduced in Section~\ref{Sec:Preliminaries}, this section is dedicated to proving Theorem~\ref{thm: uniqueness in N}. This theorem establishes a powerful comparison principle for functions in $\Ker(H)$, offering a significant tool for exploring uniqueness in the Cegrell classes. Prior to presenting the proof, we need some preliminary results.

\begin{definition}\label{def: singular}
  Given plurisubharmonic functions $u,v$ in $\Omega$, we say that $u$ \emph{is more singular than} $v$ if, for any compact set $K \Subset \Omega$, there exists a constant $C_K$ such that $u \leq v + C_K$ on $K$. If $u$ is more singular than $v$, then we denote this by $u \preceq v$. We say that $u$ and $v$ have the \emph{same singularities} if $u \preceq v$ and $v \preceq u$, denoted $u \simeq v$.
\end{definition}

The following result follows directly from \cite[Lemma 4.1]{ACCP09}:

\begin{lemma}\label{lem: ACCP09 lemma 4.1}
    If $u,v \in \E$ and $u \simeq v$, then $\mu_s(u) = \mu_s(v)$.
\end{lemma}
The converse of it is not true, even assuming $u \preceq v$. It turns out however that the difference $u-v$ is not very singular:

\begin{lemma}\label{prop: P(u-v) is in Ea}
    Assume $u,v \in \E$, $u \preceq v$, and
    \[
    {\bf 1}_{\{u = -\infty\}}(dd^c u)^n = {\bf 1}_{\{v = -\infty\}}(dd^c v)^n.
    \]
    Then $(dd^c P(u-v,0))^n \leq \mu_r(u)$. In particular, $P(u-v,0) \in \E^a$.
\end{lemma}

\begin{proof}
For $t > 0$, define $\varphi_t = \max(u, v - t)$. Then, $\varphi_t \in \E$ and $\varphi_t \simeq v$, hence by Lemma~\ref{lem: ACCP09 lemma 4.1},
    \[
    {\bf 1}_{\{\varphi_t = -\infty\}}(dd^c \varphi_t)^n = {\bf 1}_{\{v = -\infty\}}(dd^c v)^n = {\bf 1}_{\{u = -\infty\}}(dd^c u)^n.
    \]
As $\varphi_t \searrow u$ when $t \to +\infty$, \cite{Ceg04} implies that $(dd^c \varphi_t)^n$ weakly converges to $(dd^c u)^n$. From this weak convergence,
    \[
    (dd^c \varphi_t)^n = \mu_r(\varphi_t) + {\bf 1}_{\{u = -\infty\}}(dd^c u)^n \rightarrow \mu_r(u) + {\bf 1}_{\{u = -\infty\}}(dd^c u)^n,
    \]
it follows that $\mu_r(\varphi_t)$ weakly converges to $\mu_r(u)$.

    Next, define $w_t = P(\varphi_t - v, 0)$. Then $w_t \in \E \cap L^{\infty}(\Omega)$, and $w_t + v \leq \varphi_t$, with equality on the contact set $D_t = \{w_t + v = \varphi_t\}$. By Theorem~\ref{thm: maximum principle} and Theorem~\ref{thm: envelope},
    \[
    (dd^c w_t)^n = {\bf 1}_{D_t}(dd^c w_t)^n \leq {\bf 1}_{D_t} \mu_r(w_t + v) \leq {\bf 1}_{D_t} \mu_r(\varphi_t) \leq \mu_r(\varphi_t).
    \]
Since $w_t \searrow w = P(u-v, 0) \in \E$, it follows from \cite{Ceg04} and the weak convergence $(dd^c w_t)^n \to \mu_r(u)$ that $(dd^c w)^n \leq \mu_r(u)$, hence $w \in \E^a$.
\end{proof}

Utilizing a similar approach, we arrive at the subsequent result:

\begin{lemma}\label{prop : P(u-v)}
Assume $u, v \in \E$, $u \preceq v$, and $(dd^c u)^n \leq (dd^c v)^n$. Then
    \[
    (dd^c P(u-v,0))^n = 0.
    \]
\end{lemma}

\begin{proof}
Given $(dd^c u)^n \leq (dd^c v)^n$ and $u \preceq v$, \cite[Lemma 4.1]{ACCP09} implies
\[
 {\bf 1}_{\{v=-\infty\}}(dd^c v)^n \leq {\bf 1}_{\{u=-\infty\}}(dd^c u)^n \leq  {\bf 1}_{\{u=-\infty\}}(dd^c v)^n = {\bf 1}_{\{v=-\infty\}}(dd^c v)^n.
\]
The last equality above follows from the fact that $\mu_r(v)$ puts no mass on the pluripolar set $\{u=-\infty\}$ and $ {\bf 1}_{\{v=-\infty\}}  {\bf 1}_{\{u=-\infty\}}= {\bf 1}_{\{v=-\infty\}}$, since $u\preceq v$. We thus have
\[
{\bf 1}_{\{u=-\infty\}}(dd^c u)^n = {\bf 1}_{\{v=-\infty\}}(dd^c v)^n.
\]
Setting $w = P(u-v,0)$ and defining $D = \{w + v = u\}$, Lemma~\ref{prop: P(u-v) is in Ea} ensures $w \in \E^a$. Since $w+v\leq u$ with equality on $D$, Theorem \ref{thm: maximum principle} yields
	\[
	{\bf 1}_D\mu_r(w) + {\bf 1}_D \mu_r(v) \leq {\bf 1}_D \mu_r(u)  \leq {\bf 1}_D \mu_r(v).
	\]
Hence, ${\bf 1}_D \mu_r(w) = 0$. Furthermore, Theorem~\ref{thm: envelope} gives that $\mu_r(w)$ is supported on $D$, leading to $\mu_r(w) = 0$, which completes the proof.
\end{proof}

By synthesizing the discussions above, we establish the main result of this section, underscoring uniqueness in $\Ker$:

\begin{theorem}\label{thm: uniqueness in N}
    Assume $H \in \E$, $u, v \in \mathcal{N}(H)$, and $u \preceq v$.
    \begin{enumerate}
        \item If $(dd^c u)^n \leq (dd^c v)^n$, then $u \geq v$.
        \item If $(dd^c u)^n = (dd^c v)^n$, then $u = v$.
    \end{enumerate}
\end{theorem}

\begin{proof}
Let $\varphi \in \mathcal{N}$ such that $\varphi + H \leq u \leq H$. Then, $u - v \geq \varphi$, implying $P(u - v, 0) \in \mathcal{N}$. Moreover, Lemma~\ref{prop : P(u-v)} ensures that $(dd^c P(u-v,0))^n = 0$. This, together with \cite[Theorem 3.6]{ACCP09}, leads to $P(u-v,0) = 0$, thereby establishing $u \geq v$.

The second statement directly follows from the first. If $(dd^c u)^n = (dd^c v)^n$, then $u \geq v$, implying $u \simeq v$. Reversing the roles of $u$ and $v$ yields $v \geq u$, thus concluding $u = v$.
\end{proof}

Theorem~\ref{thm: uniqueness in N} was previously known under the condition $\int_{\Omega} (-w) (dd^c u)^n < +\infty$ for some $w \in \E_0$, $w < 0$; see \cite[Theorem 3.6]{ACCP09}. As illustrated in \cite[Example 5.3]{Ceg08}, there exists a function $u \in \Ker \cap L^{\infty}$ for which $\int_{\Omega} (-w)(dd^c u)^n = +\infty$ for all $w \in \PSHM{\Omega}$, $w < 0$.

Thanks to the uniqueness result above, we present the following theorem:

\begin{theorem}\label{thm: decomposition}
For any $u \in \Ker$, there exist unique functions $u_r, u_s \in \Ker$ satisfying the following conditions:
    \begin{enumerate}
        \item $u \leq u_r$, $u \leq u_s$;
        \item $(dd^c u_r)^n = \mu_r(u) = {\bf 1}_{\{u > -\infty\}}(dd^c u)^n$;
        \item $(dd^c u_s)^n = \mu_s(u) = {\bf 1}_{\{u = -\infty\}}(dd^c u)^n$.
    \end{enumerate}
Moreover, $u_r + u_s \leq u$.
\end{theorem}

\begin{proof}
By Theorem 4.14 (2) in~\cite{ACCP09}, there exists $u_r, u_s\in \E$ satisfying the three conditions in the theorem. It follows from Lemma~\ref{prop: P(u-v) is in Ea} that $P(u-u_s,0)\in \E^a$ and
\[
(dd^c P(u-u_s,0))^n\leq (dd^c u_r)^n.
\]
By~\cite[Corollary 3.2]{ACCP09}, we get the uniqueness of $u_r$ and we also have $u_r \leq P(u-u_s,0)$, hence $u_r+u_s\leq u$.

Assume now that $v\in \Ker$ is such that $u\leq v$ and $(dd^c v)^n = (dd^c u_s)^n={\bf 1}_{\{u=-\infty\}}(dd^c u)^n$. Then $w=P(u_s,v)\in \Ker$ and, by Lemma 4.1 in \cite{ACCP09}, since $u\leq w\leq \min(u_s,v)$,
\[
{\bf 1}_{\{w=-\infty\}}(dd^c w)^n = {\bf 1}_{\{v=-\infty\}} (dd^c v)^n = {\bf 1}_{\{u_s=-\infty\}} (dd^c u_s)^n.
\]
We get $(dd^c w)^n=(dd^c v)^n=(dd^c u_s)^n$ because these measures are supported on pluripolar sets.
Theorem \ref{thm: uniqueness in N} then ensures that $v=w=u_s$, finishing the proof.
\end{proof}

\section{Rooftops, Idempotency and Connectivity}\label{sec:RIC}

By employing Theorem~\ref{thm: uniqueness in N}, this section not only aims to prove our result on geodesic connectivity (Theorem~\ref{thm: geodesic connectivity}) but also significantly advance the development of rooftop techniques. We commence in Section~\ref{subsec1} by revisiting the necessary definitions of the asymptotic rooftop envelope $P[u](v)$ and the Green-Poisson residual function $g_u$, commonly referred to as the residual function. Theorem~\ref{thm: MA measure of asymptotic envelope} is then presented, offering insights into the Monge-Ampère measure associated with these residual functions.

 The \emph{idempotency property} for residual functions, $g_{g_u}=g_u$, was
conjectured in~\cite{Rash22} to hold for all $u\in\mathcal{PSH}^-(\Omega)$. In Theorem~\ref{thm:idempotenceH}, we provide an affirmative answer in $\Ker$, utilizing Theorem~\ref{thm: uniqueness in N} and Theorem~\ref{thm: MA measure of asymptotic envelope}. We then proceed to the discussion of the rooftop equality in Section~\ref{subsec3}. Recall that the \emph{rooftop equality} holds for a plurisubharmonic function $u$ if
\[
P[u](v) = P(g_u,v), \quad \text{ for all } v \in \mathcal{PSH}^-(\Omega).
\]
This property, verified in \cite{Ras23} for functions in $\F$, is substantially generalized in Theorem~\ref{thm: RE conj N}.

Subsequently, in Section~\ref{subsec4}, we investigate the geodesic connectivity of plurisubharmonic functions. Building on the foundation established by Theorem~\ref{thm: uniqueness in N} and Theorem~\ref{thm: RE conj N}, we complete the proof of Theorem~\ref{thm: geodesic connectivity}.

We conclude this overview with the following conjecture, emphasizing  that if Conjecture~\ref{conj} holds true, it would affirm the idempotency property for all functions in $\mathcal{PSH}^-(\Omega)$. To see this, let $u_t$ be the largest plurisubharmonic geodesic segment lying below $u$ and $0$, then $u_0=g_u$. This is a direct consequence of the geodesic connectivity criterion in \cite{Rash22}  since $P[u](0)=g_u$. Hence, any $u\in \PSHM{\Omega}$ can be connected to $g_u$ by a plurisubharmonic geodesic (see also \cite[Corollary 8.2]{Rash22}). Thus, if the Conjecture~\ref{conj} holds, then so is the idempotency conjecture $g_u=g_{g_u}$.

\medskip

\begin{conjecture}\label{conj}
Let $u_0, u_1 \in \E$. Then $u_0$ and $u_1$ can be connected by a plurisubharmonic geodesic if and only if $g_{u_0} = g_{u_1}$.
\end{conjecture}

\subsection{Asymptotic rooftops, residual functions, and idempotency}\label{subsec1}

Given $u,v \in \mathcal{PSH}^-(\Omega)$, the \emph{asymptotic rooftop envelope} $P[u](v)$ is defined as
\[
P[u](v) = \left( \lim_{C \to +\infty} P(u + C, v) \right)^*.
\]
For the case where $v = 0$, we denote $g_u = P[u](0)$ and refer to it as the \emph{Green-Poisson residual function} of $u$, or simply the \emph{residual function} of $u$. We will use repeatedly that for any $u, v, w \in \PSHM{\Omega}$, we have
\[
P[u + w](v) \geq P[u](v) + g_w .
\]
In particular, $g_{u + w} \geq g_u + g_w$, and if $g_w = 0$, then $g_{u + w} = g_u$. The condition $g_w = 0$ means that $w$ does not possess strong singularities, neither in $\Omega$ nor on the boundary $\partial\Omega$.

By construction, it holds that $u \le g_u$. However, their singularities may not coincide, implying that the relation $u \simeq g_u$ may not hold. Nevertheless, the discrepancy between the singularities is minimal. In this context, $g_u$ can be viewed as a singular skeleton of $u$.

\begin{theorem}\label{thm: MA measure of asymptotic envelope}
	For all $u\in \PSHM{\Omega}$, we have
	$\mu_r(g_u) =0$. If $u\in \mathcal E$, then $(dd^c P(u-g_u))^n\leq \mu_r(u)$, so $P(u-g_u)\in \E^a$, and $(dd^c g_u)^n = \mu_s(u)$.
\end{theorem}
\begin{proof}
Fix $C>0$ and let $[u_j]$ be a sequence of bounded plurisubharmonic functions decreasing to $u$.
Set $\varphi_{j,C}^t= \max(\varphi_{j,C},-t)$, $\varphi_{j,C}=P(u_j+C,0)$, $\varphi_C=P(u+C,0)$.
According to Corollary \ref{cor: NP of rooftop} and \eqref{eq: NP plurifine property}, we have
\[
{\bf 1}_{\{\varphi_{j,C}>-t;\; u_j>-C\}}(dd^c \varphi_{j,C}^t)^n=0,
\]
leading to
	 \[
	 f_{j,C,\varepsilon}^t (dd^c \varphi_{j,C}^t)^n =0,
	 \]
	where
	 \[
	 f_{j,C,\varepsilon}^t = \frac{\max(\varphi_{j,C}+t,0)}{\max(\varphi_{j,C}+t,0)+ \varepsilon} \times \frac{\max(u_j+C,0)}{\max(u_j+C,0)+ \varepsilon},
	 \]
are quasi-continuous, uniformly bounded (with values in $[0,1]$), and converge in capacity to
	 \[
	  f_{C,\varepsilon}^t = \frac{\max(\varphi_{C}+t,0)}{\max(\varphi_{C}+t,0)+ \varepsilon} \times \frac{\max(u+C,0)}{\max(u+C,0)+ \varepsilon},
	 \]
  as $j \to +\infty$. Thus, using Xing's theorem~\cite{Xing00,GZbook} we obtain
	  \[
	 f_{C,\varepsilon}^t (dd^c \max(\varphi_C,-t))^n  =0.
	 \]
For $C_0 < C$, we then have
\[
\frac{\max(\varphi_{C}+t,0)}{\max(\varphi_{C}+t,0)+ \varepsilon} \times \frac{\max(u+C_0,0)}{\max(u+C_0,0)+ \varepsilon}(dd^c \max(\varphi_{C},-t))^n=0.
\]
Letting $C\to +\infty$, we arrive at
\[
\frac{\max(g_u+t,0)}{\max(g_u+t,0)+ \varepsilon} \times \frac{\max(u+C_0,0)}{\max(u+C_0,0)+ \varepsilon}(dd^c \max(g_u,-t))^n=0.
\]
We finally let $\varepsilon\to 0$ and then $C_0\to +\infty$ to get
\[
{\bf 1}_{\{g_u>-t\}}(dd^c \max(g_u,-t))^n=0,
\]
which yields $\mu_r(g_u)=0$.

We next assume $u\in \E$. To see that $\mu_s(u)=\mu_s(g_u)$, one can use \cite{Ceg18}. We provide an alternative proof using the plurisubharmonic envelopes.
We set $w_t=P(u-v_t)=P(u-v_t,0)$. Then $w_t\in \PSHM{\Omega}\cap L^{\infty}(\Omega)$ and $w_t\searrow P(u-g_u)$. Arguing as in the proof of Lemma~\ref{prop: P(u-v) is in Ea} we see that
\[
(dd^c w_t)^n \leq \mu_r(u),
\]
hence $(dd^c P(u-g_u))^n\leq \mu_r(u)$. In particular $P(u-g_u)\in \E^a$. Thus, from
\[
g_u+P(u-g_u) \leq u \leq g_u,
\]
and Lemma \ref{lem: ACCP09 lemma 4.1} we conclude that $\mu_s(g_u)=\mu_s(u)$.
\end{proof}

\begin{remark}
If we assume that $u\in \E$, in the first statement of Theorem~\ref{thm: MA measure of asymptotic envelope},  then the proof can be simplified as follows.
For each $t>0$ we set $v_t=P(u+t,0)$. Then $v_t\nearrow g_u$ almost everywhere in $\Omega$, hence $\mu_r(v_t)$ weakly converges to $\mu_r(g_u)$ as follows from Theorem \ref{thm: lsc of NP}.
It follows from Corollary \ref{cor: NP of rooftop} that
    \[
    \mu_r (v_t) \leq {\bf 1}_{\{v_t=u+t\}}  \mu_r(u) \leq {\bf 1}_{\{u\leq -t\}} \mu_r(u) \to 0
    \]
as $t\to +\infty$, giving $\mu_r(g_u)=0$.  Here, we use the fact that $\mu_r(u)$ is a positive non-pluripolar measure.
\end{remark}

\begin{remark}
The relation $P(u - g_u) \in \E^a$ means that the function $u \in \E$ differs from $g_u$ by a plurisubharmonic function that lacks strong singularities. However, this condition does not universally apply to all plurisubharmonic functions, as demonstrated by the following example. Consider the function $u(z) = u(z_1, z') = \log|z_1|$ in the unit ball $\B$. For this function, we have
\[
g_{u}(z) = \log\frac{|z_1|}{\sqrt{1 - |z'|^2}}, \quad z = (z_1, z')
\]
(see \cite[Example 4.7.1]{Rash22}). Consequently, for $w = P(u - g_u)$, it holds that $w \le u - g_u = \frac{1}{2}\log(1 - |z'|^2)$. Applying the maximum principle to the slices $\{z \in \B : z_1 = z_1^0\}$, we find $w \le \log|z_1|$, which leads to $g_w = g_u$.
\end{remark}

Next, we establish that the idempotency property for residual functions, $g_{g_u} = g_u$, is valid, particularly for functions in $\Ker$. This result extends the findings of~\cite[Theorem 3.6]{Rash22}.

\begin{theorem}\label{thm:idempotenceH}
Let $u\in \Ker(H)$, $H\in\E$. If $H$ satisfies $g_H=g_{g_H}$, then $g_u=g_{g_u}$. In particular, if $g_u\in \Ker$ then $g_{g_u}=g_u$.
\end{theorem}
\begin{proof}
Consider $w \in \Ker$ satisfying $u \geq H + w$. Noting that
\[
g_{u+w} \geq g_u + g_w \geq g_u + w,
\]
we deduce that $g_u \in \Ker(g_H)$, $g_{g_u} \in \Ker(g_{g_H})$. Theorem~\ref{thm: MA measure of asymptotic envelope} assures that $(dd^c g_u)^n = (dd^c g_{g_u})^n = \mu_s(u)$. Thanks to Theorem~\ref{thm: uniqueness in N}, we conclude that $g_u = g_{g_u}$.
\end{proof}

\subsection{Rooftop equality}\label{subsec3}

In the next section, we address plurisubharmonic geodesics and connectivity, where the rooftop equality is essential. That is, for any plurisubharmonic function $u$, the following holds:
\begin{equation}\label{eq:RTi}
P[u](v) = P(g_u,v), \quad \text{for all } v \in \mathcal{PSH}^-(\Omega).
\end{equation}
In Theorem~\ref{thm: RE conj N}, we prove that if $H_1 \in \E$ and $H_2 \in \mathcal{PSH}^-(\Omega)$ satisfy \eqref{eq:RTi}, then this relation extends to all $u \in \Ker(H_1)$ and $v \in \Ker(H_2)$. However, we must first establish some auxiliary results.

\begin{lemma}\label{lem: MA of Puv v not in E}
	Let $u,v\in \PSHM{\Omega}$, and set $w=P(u,v)$. Then
	\[
	{\bf 1}_{\{-t<w<v\}} (dd^c \max(w,-t))^n \leq {\bf 1}_{\{-t<u\}} (dd^c \max(u,-t))^n, \; t>0.
	\]
\end{lemma}
\begin{proof}
For $t > 0$, define $u_t = \max(u, -t)$, $v_t = \max(v, -t)$, and $w_t = P(u_t, v_t)$. Given $s > t > 0$, we apply Corollary~\ref{cor: NP of rooftop} and \eqref{eq: NP plurifine property} to obtain
	 \[
	 f_{s,\varepsilon}^t (dd^c \max(w_s,-t))^n \leq  {\bf 1}_{\{-t<w_s<v_s\}} (dd^c u_s)^n \leq {\bf 1}_{\{-t<u\}}(dd^c u_t)^n,
	 \]
where the function
	 \[
	 f_{s,\varepsilon}^t = \frac{\max(w_s+t,0)}{\max(w_s+t,0)+ \varepsilon} \times \frac{\max(v_s-w_s,0)}{\max(v_s-w_s,0)+ \varepsilon},
	 \]
is quasi-continuous, uniformly bounded with values in $[0,1]$, and converges in capacity to
	 \[
	  f_{\varepsilon}^t = \frac{\max(w+t,0)}{\max(w+t,0)+ \varepsilon} \times \frac{\max(v-w,0)}{\max(v-w,0)+ \varepsilon},
	 \]
as $s \to +\infty$.

Applying Xing's theorem~\cite{Xing00,GZbook}, we deduce
	  \[
	 f_{\varepsilon}^t (dd^c \max(w,-t))^n  \leq {\bf 1}_{\{-t<u\}}(dd^c u_t)^n.
	 \]
As $\varepsilon \to 0^+$, we conclude the proof with the desired result.
\end{proof}

\begin{corollary}
	\label{cor: MA of Puv v not in E}
	Let $u,v\in \PSHM{\Omega}$, and set $w=P[u](v)$. Then
	\[
	{\bf 1}_{\{-t<w<v\}} (dd^c \max(w,-t))^n =0, \; t>0.
	\]
\end{corollary}
\begin{proof}
Let us set $u_C = P(u + C, 0)$ and $\varphi_C = P(u_C, v)$. Notice that $\varphi_C = P(u + C, v)$. Utilizing Lemma \ref{lem: MA of Puv v not in E} with $u$ replaced by $u_C$, we derive
	\[
	{\bf 1}_{\{-t<\varphi_C<v\}} (dd^c \max(\varphi_C,-t))^n \leq {\bf 1}_{\{-t<u_C\}}(dd^c \max(u_C,-t))^n, \; t>0.
	\]
Following the argument in the proof of Lemma \ref{lem: MA of Puv v not in E} with
 \[
 f_{C,\varepsilon}^t = \frac{\max(\varphi_C+t,0)}{\max(\varphi_C+t,0)+\varepsilon} \times \frac{\max(v-\varphi_C,0)}{\max(v-\varphi_C,0)+\varepsilon},
 \]
we reach the inequality
\[
	{\bf 1}_{\{-t<w<v\}} (dd^c w,-t))^n \leq {\bf 1}_{\{-t<g_u\}}(dd^c \max(g_u,-t))^n, \; t>0.
	\]
To conclude, we apply Theorem \ref{thm: MA measure of asymptotic envelope} to obtain the desired result.
\end{proof}

\begin{lemma}\label{lem: MA of w is zero}
    Let $u \in \E$, $v\in \PSHM{\Omega}$, and set $w=P(P[u](v)-P(g_u,v))$. Then $(dd^c w)^n=0$.
\end{lemma}
\begin{proof}
For simplicity, denote $f_t = \max(f, -t)$. Define
\[
\varphi = P[u](v), \quad \psi = P(g_u, v), \quad w = P(\varphi - \psi).
\]
Since,
\[
P[u](v) \geq P(u, v) \geq P(u - g_u) + P(g_u, v),
\]
we obtain $w \geq P(u - g_u)$. By Theorem~\ref{thm: MA measure of asymptotic envelope}, it follows that $w \in \E^a$.

Considering the case $w + \psi_t \leq \varphi_t$, equality holds on the set $D_t = \{w + \psi_t = \varphi_t\}$. Specifically, if $\psi(x) > -t$, then $w(x) + \psi_t(x) = w(x) + \psi(x) \leq \varphi(x) \leq \varphi_t(x)$. If $\psi(x) \leq -t$, then $w(x) + \psi_t(x) = w(x) - t \leq \varphi_t(x)$ since $w \leq 0$. From Theorem~\ref{thm: maximum principle}, we have
\[
{\bf 1}_{D_t} (dd^c w)^n + {\bf 1}_{D_t} (dd^c \psi_t)^n \leq {\bf 1}_{D_t} (dd^c \varphi_t)^n.
\]
Multiplying with ${\bf 1}_{\{\varphi =v\}}$, and noting that $\{\varphi =v\} \subset \{\varphi_t= \psi_t\}$, we obtain
\begin{flalign*}
	{\bf 1}_{D_t \cap \{\varphi =v\}} (dd^c w)^n + {\bf 1}_{D_t\cap \{\varphi =v\}} (dd^c \psi_t)^n & \leq {\bf 1}_{D_t \cap \{\varphi =v\}}  (dd^c \varphi_t)^n\\
	&  \leq  {\bf 1}_{D_t\cap \{\varphi =v\}} (dd^c \psi_t)^n.
\end{flalign*}
Here, we have used  $\varphi_t\leq \psi_t$, and Theorem \ref{thm: maximum principle} to obtain
\[
{\bf 1}_{\{\varphi_t=\psi_t\}} (dd^c \varphi_t)^n\leq {\bf 1}_{\{\varphi_t=\psi_t\}} (dd^c \psi_t)^n.
\]
It thus follows that ${\bf 1}_{D_t \cap \{\varphi =v\}} (dd^c w)^n=0$. By Corollary \ref{cor: MA of Puv v not in E}, we also have
\[
{\bf 1}_{D_t \cap \{-t<\varphi <v\}} (dd^c \varphi_t)^n=0.
\]
Combining these we then have
\[
{\bf 1}_{D_t \cap \{\varphi>-t\}} (dd^c w)^n=0.
\]
Let $D= \{w+\psi=\varphi\}$.
Since $\{\varphi>-\infty\} \cap D$ is contained in the union of $D_j \cap \{\varphi>-j\}$, $j\in \mathbb N$, it follows that
\[
\int_{\{\varphi>-\infty\} \cap D} (dd^c w)^n=0.
\]
Since $w\in \E^a$ and $(dd^c w)^n$ is supported on $D$ (by Theorem \ref{thm: envelope}), we infer that $(dd^c w)^n=0$, finishing the proof.
\end{proof}

\begin{theorem}
    \label{thm: RE conj N}
	Assume $H_1\in \E$ and $H_2\in \PSHM{\Omega}$ satisfy
 \[P[H_1](H_2)=P(g_{H_1},H_2).\]
 Then the same relation holds for all  $u\in \Ker(H_1)$, $v\in \Ker(H_2)$, i.e.
 \[P[u](v)=P(g_u,v).\]
\end{theorem}

\begin{remark}
The equality $P[H_1](H_2) = P(g_{H_1}, H_2)$ is satisfied, particularly when there exists $w \in \E$ such that $H_2 + w \leq H_1$ and $g_w = 0$. This is seen through
\[
P[H_1](H_2) \geq P[H_2 + w](H_2) \geq H_2 + g_w = H_2 \geq P(g_{H_1}, H_2).
\]
\end{remark}

 In other words, Theorem~\ref{thm: RE conj N}  says that if the rooftop equality holds for $H_1,H_2$, then it holds for all $u\in \Ker(H_1)$, $v\in \Ker(H_2)$. In particular, the rooftop equality holds in $\Ker$.
\begin{proof}
For notational convenience, let us write
\[
\varphi=P[u](v), \; \psi=P(g_u,v), \;  w=P(\varphi-\psi).
\]
Then $\varphi\leq \psi$, hence $w\leq 0$. We also have $w\geq P(u-g_u)$. By the assumption $u\in \Ker(H_1)$, $v\in \Ker(H_2)$,  we deduce $P(u-H_1)\in \Ker$ and $P(v-H_2)\in \Ker$, hence
\begin{flalign*}
P(u+C,v)  &\geq P(P(u-H_1)+H_1+C,P(v-H_2)+H_2) \\
&\geq P(u-H_1) + P(v-H_2)+ P(H_1+C,H_2).
\end{flalign*}
Letting $C\nearrow+\infty$, we arrive at
\begin{flalign*}
    P[u](v)  &\geq   P(u-H_1) + P(v-H_2) + P[H_1](H_2) \\
    &= P(u-H_1)+ P(v-H_2)+ P(g_{H_1},H_2)\\
    &\geq P(u-H_1)+ P(v-H_2) + P(g_u,v).
\end{flalign*}
From this we get $w\geq P(u_1-H_1)+P(v-H_2)$, hence $w\in \Ker$. Since, by Theorem~\ref{thm: MA measure of asymptotic envelope}, $P(u-g_u)\in \E^a$, we thus infer  $w\in \Ker^a$. It  follows from Lemma \ref{lem: MA of w is zero} that $(dd^c w)^n=0$, hence by uniqueness in $\Ker^a$ (see \cite[Corollary 3.2]{ACCP09}), $w=0$, ultimately giving $\varphi=\psi$.
\end{proof}

Using the same ideas as above we now prove that the rooftop envelope commutes with the Green-Poisson residual operator $g$, provided that the same identity holds for the boundary values.

\begin{theorem}\label{thm: green of rooftop new}
    Let $H_1, H_2 \in \E$ be such that
    \[
    g_{P(H_1, H_2)} = P(g_{H_1}, g_{H_2}).
    \]
    Then, for any $u \in \Ker(H_1)$ and $v \in \Ker(H_2)$, it holds that $g_{P(u, v)} = P(g_u, g_v)$.
\end{theorem}

The relation $g_{P(H_1,H_2)} = P(g_{H_1},g_{H_2})$ is trivially satisfied when $H_1 = H_2 = 0$. Consequently, this leads to the identity $g_{P(u,v)} = P(g_u,g_v)$ being valid for all $u, v \in \Ker$.

\begin{proof}
    Setting $w=P(u-H_1)+P(v-H_2)\in \Ker$, we have
    \[
    P(u,v) \geq w+ P(H_1,H_2),
    \]
    hence
    \[
    g_{P(u,v)} \geq g_w+ g_{P(H_1,H_2)}=g_w+P(g_{H_1},g_{H_2})\geq g_w+ P(g_u,g_v),
    \]
    which yields $g_{P(u,v)}\in \Ker(P(g_u,g_v))$.
    On the other hand, we have
    \[
    P(g_u,g_v)+P(u-g_u)+P(v-g_v)\leq P(u,v)\leq P(g_u,g_v).
    \]
   Note also that, by Theorem~\ref{thm: MA measure of asymptotic envelope}, $P(u-g_u)\in \E^a$ and $P(v-g_v)\in \E^a$. It thus follows from \cite[Lemma 4.4]{ACCP09} and Theorem \ref{thm: MA measure of asymptotic envelope} that
\[
\mu_s(P(g_u,g_v))=\mu_s(P(u,v))=\mu_s(g_{P(u,v)}).
\]
By Corollary \ref{cor: NP of rooftop} and Theorem \ref{thm: MA measure of asymptotic envelope} we have
\[
\mu_r(P(g_u,g_v))=\mu_r(g_{P(u,v)})=0,
\]
therefore
\[
(dd^c P(g_u,g_v))^n=(dd^c g_{P(u,v)})^n.
\]
   Since $P(g_u,g_v)\ge g_{P(u,v)}$, from Theorem \ref{thm: uniqueness in N} we thus obtain $P(g_u,g_v)=g_{P(u,v)}$.
\end{proof}

\subsection{Geodesic connectivity}\label{subsec4}

We start by revisiting the concept of plurisubharmonic geodesics. Consider a curve of plurisubharmonic functions $t \mapsto u_t$ for $t \in [0,1]$ and $u_t \in \PSH{\Omega}$. This curve defines a function $U$ on $\Omega \times A(1,e)$ as
\[
U(x,z) = u_{\log |z|}(x), \quad x \in \Omega, \; z \in A(1,e),
\]
where $A(1,e)$ denotes the annulus in $\mathbb{C}$ with radii $1$ and $e$. A \emph{subgeodesic segment} $u_t$ is characterized by the plurisubharmonicity of $U$ in $\Omega \times A(1,e)$.

Let $\mathcal{S}(u_0,u_1)$ represent the set of all subgeodesic segments beneath $u_0$ and $u_1$, satisfying
\[
\limsup_{t\to 0} u_t \leq u_0 \quad \text{and} \quad \limsup_{t\to 1} u_t \leq u_1.
\]
The function $u_0+u_1$ is an element of $\mathcal{S}(u_0,u_1)$. Due to convexity, for each $U \in \mathcal{S}(u_0,u_1)$, we have
\[
u(x,z) \leq (1-\log |z|)u_0(x) + \log |z| u_1(x),
\]
where the right-hand side is upper semicontinuous. Consequently, the upper semicontinuous regularization of
\[
(x,z) \mapsto \sup \{U(x,z) \; : \; U \in \mathcal{S}(u_0,u_1)\}
\]
forms a plurisubharmonic subgeodesic below $u_0$ and $u_1$. Remarkably, this function, identified as the largest plurisubharmonic geodesic segment below $u_0$ and $u_1$, does not require further regularization.

\begin{definition}
    Define $u_t, 0 \leq t \leq 1$, as the largest plurisubharmonic geodesic beneath $u_0$ and $u_1$. We say $u_0$ and $u_1$ are \emph{connectable} by a plurisubharmonic geodesic if
    \[
    \lim_{t\to 0} u_t = u_0 \quad \text{and} \quad \lim_{t\to 1} u_t = u_1,
    \]
with the limit understood in terms of $L^1_{\rm loc}$ convergence or capacity convergence, which are equivalent as shown in \cite{Dar15,Rash22}.
\end{definition}

As established in \cite[Theorem 5.2]{Dar17AJM} and \cite[Theorem 8.1]{Rash22}, two functions $u_0, u_1 \in \PSHM{\Omega}$ are geodesically connectable if and only if
\begin{equation}\label{eq: Darvas criterion}
    P[u_0](u_1) = u_1 \quad \text{and} \quad P[u_1](u_0) = u_0.
\end{equation}
We shall now use these conditions for the class $\Ker$, relating them to the residual functions $g_{u_0}$ and $g_{u_1}$ to arrive at our result on geodesic connectivity.

\begin{theorem}\label{thm: geodesic connectivity}
  Given $H_0,H_1 \in \E$ that are connectable by a plurisubharmonic geodesic, and $u_0\in\Ker(H_0)$, $u_1\in\Ker(H_1)$, then $u_0$ and $u_1$ can be connected by a  plurisubharmonic geodesic segment if and only if
\begin{equation}
    \label{eq: geo connectivity subsec}
    u_0\leq g_{u_1}\quad \text{ and }\quad u_1\leq g_{u_0}.
\end{equation}
   In particular, if $g_{H_0}=g_{H_1}$, then \eqref{eq: geo connectivity subsec} is equivalent to $g_{u_0}=g_{u_1}$.
\end{theorem}
\begin{proof}
Starting with (\ref{eq: Darvas criterion}), we observe that
\[
P[H_0](H_1) = H_1 = P(g_{H_0},H_1) \quad \text{and} \quad P[H_1](H_0) = H_0 = P(H_0,g_{H_1}).
\]
Invoking Theorem \ref{thm: RE conj N}, it follows that
\[
P[u_0](u_1) = P(g_{u_0},u_1) \quad \text{and} \quad P[u_1](u_0) = P(u_0,g_{u_1}).
\]
Thus the condition (\ref{eq: geo connectivity subsec}) is equivalent to (\ref{eq: Darvas criterion}), which is equivalent to the connectivity of $u_0$ and $u_1$ through a plurisubharmonic geodesic.

Furthermore, if $g_{H_0} = g_{H_1}$, then $g_{u_0}, g_{u_1},$ and $v = P(g_{u_0}, g_{u_1})$, all belong to $\Ker(g_{H_0})$. Since $u_0 \leq v \leq g_{u_0}$, we deduce that
\[
\mu_s(v) = \mu_s(u_0) = \mu_s(g_{u_0}).
\]
Corollary~\ref{cor: NP of rooftop} implies $\mu_r(v) = 0$, leading to $(dd^c v)^n = (dd^c g_{u_0})^n$. According to Theorem \ref{thm: uniqueness in N}, this results in $v = g_{u_0}$. Interchanging $u_0$ and $u_1$, we similarly conclude $v = g_{u_1}$, ultimately showing that $g_{u_0} = g_{u_1}$.
\end{proof}

\section{Boundary Values in the Cegrell Classes}\label{Sec: BV}

Compared to the well-explored domain of boundary values for subharmonic, convex, or holomorphic functions, the investigation into the boundary values of plurisubharmonic functions remains less developed. Cegrell was inspired by Riesz’s decomposition theorem, which asserts that any non-positive subharmonic function on a bounded domain can be decomposed into the sum of a Green potential and a harmonic function. The smallest harmonic majorant of the Green potential is zero, and the behavior of the harmonic function near the boundary determines it. Consequently, the boundary values of a subharmonic function can be understood as the harmonic function in Riesz’s decomposition theorem. However, direct generalization of this decomposition to pluripotential theory is not possible. Instead, as elaborated in~\cite{Ceg98,Ceg08} (see Section~\ref{Sec:Preliminaries} for details), we examine the Cegrell class $\Ker(H)$, which consists of functions with boundary values  $H$, satisfying the inequality:
\begin{equation}
H \geq u \geq \varphi + H,
\end{equation}
for some function $\varphi\in\Ker$.

Theorem~\ref{bvinN} ensures that a function $u \in \E$, satisfying specific integrability criteria, is in $\Ker(\tilde u)$, which means it possesses boundary values represented by $\tilde u$.

\begin{theorem}\label{bvinN}
Assume $u\in\E$ and there exists a function $w\in\mathcal{PSH}^-(\Omega)$, $w< 0$, such that
\[
\int_{\Omega}(-w)(dd^cu)^n<+\infty.
\]
Then $u\in \Ker(\tilde u)$.
\end{theorem}

We initiate with the proof of an existence and uniqueness result for the complex Monge-Ampère equation. The existence part was first established in~\cite[Proposition 4.3]{Hiep14}. The novelty herein lies in the uniqueness part, which is achieved through Theorem~\ref{thm: uniqueness in N}.

\begin{theorem}\label{thm: sol in N}
	Assume $u\in \mathcal E$ and $w\in \mathcal E_0$ are such that $-1\leq w<0$ and
	\[
	\int_{\Omega} (-w) (dd^c u)^n <+\infty.
	\]
	Then there exists a unique $\varphi \in \mathcal N$ such that $u\leq \varphi$ and $(dd^c \varphi)^n =(dd^c u)^n$.
\end{theorem}
\begin{proof} Consider a fundamental sequence $[\Omega_j]$ of $\Omega$ and define
\[
u_j = P({\bf 1}_{\Omega_j} u).
\]
We have $u_j \in \F$ and $(dd^c u_j)^n \geq {\bf 1}_{\Omega_j} (dd^c u)^n$. According to \cite[Theorem 4.14, (2)]{ACCP09}, there exists a decreasing sequence $[\psi_j] \subset \F$ satisfying
\[
(dd^c \psi_j)^n =  {\bf 1}_{\Omega_j} (dd^c u)^n.
\]
For a fixed $j$, and for each $k > j$, \cite[Theorem 4.14, (2)]{ACCP09} ensures the existence of $v_{j,k} \in \F$ such that
\[
(dd^c v_{j,k})^n = {\bf 1}_{\Omega_k \setminus \Omega_j} (dd^c u)^n,  \; u\leq v_{j,k+1}\leq v_{j,k}.
\]
By \cite[Lemma 3.5]{ACCP09}, we have
\[
\int_{\Omega} (-v_{j,k})^n (dd^c w)^n \leq n! \int_{\Omega} (-w) (dd^c v_{j,k})^n \leq n! \int_{\Omega \setminus \Omega_j} (-w) (dd^c u)^n.
\]
As $k\to+\infty$, $v_{j,k}$ decreases to $v_j \geq u$. Hence, $v_j \in \E$, and
\begin{equation}\label{eq: sol in N construction v}
(dd^c v_j)^n =  {\bf 1}_{\Omega \setminus \Omega_j} (dd^c u)^n, \; \int_{\Omega} (-v_j)^n (dd^c w)^n \leq n! \int_{\Omega\setminus \Omega_j} (-w) (dd^c u)^n.
\end{equation}
Having $v_j$, we repeat the above procedure to define $v_{j+1,k}\geq v_j, k=1,2,...$, and thus construct an increasing sequence $[v_j]\subset \E$ satisfying \eqref{eq: sol in N construction v}. Given the finiteness of $\int_{\Omega} (-w) (dd^c u)^n$, the Lebesgue dominated convergence theorem implies $v_j \nearrow 0$ almost everywhere.

Now, considering $v_j + \psi_j \in \E$ and $(dd^c (v_j+\psi_j))^n \geq (dd^c u)^n$, \cite[Theorem 4.14 (2)]{ACCP09} guarantees the existence of $\varphi \in \E$ with $(dd^c \varphi)^n=(dd^c u)^n$ and $\varphi \geq v_j + \psi_j$. Hence,
\[
\tilde{\varphi} \geq \tilde{v_j} + \tilde{\psi_j} \geq v_j,
\]
where $\psi_j \in \F \subset \Ker$. Letting $j \to +\infty$, we conclude $\varphi \in \mathcal N$, as desired.

Finally, we aim to establish the uniqueness of $\varphi$. Suppose $\psi\in \Ker$ with $u\leq \psi$ and $(dd^c \psi)^n =(dd^c u)^n$. Define $w=P(\varphi,\psi)$; then $w\in \Ker$ and $u\leq w \leq \min(\varphi,\psi)$. By \cite[Lemma 4.1]{ACCP09}, we obtain
\[
{\bf 1}_{\{w=-\infty\}}(dd^c w)^n  = {\bf 1}_{\{u=-\infty\}} (dd^c u)^n =  {\bf 1}_{\{\varphi=-\infty\}} (dd^c \varphi)^n = {\bf 1}_{\{\psi=-\infty\}} (dd^c \psi)^n.
\]
Further, according to Corollary \ref{cor: NP of rooftop}, $\mu_r(w)\leq \mu_r(\varphi)=\mu_r(\psi)$. Consequently, we have $(dd^c w)^n \leq (dd^c \varphi)^n$ and $(dd^c w)^n\leq (dd^c \psi)^n$. Applying Theorem \ref{thm: uniqueness in N}, we conclude that $w=\varphi=\psi$.
\end{proof}

\begin{lemma} \label{thm: from sol to boundary}
If $u\in \E$, $v\in \Ker$, $u\leq v$, and $(dd^c u)^n =(dd^c v)^n$, then $u\in \Ker(\tilde{u})$.
\end{lemma}
\begin{proof}
By Lemma~\ref{prop: P(u-v) is in Ea}, the function $w=P(u-v)$ is in $\E^a$ and satisfies $(dd^c w)^n=0$, implying $\tilde{w}=w$.  Since $w+v\leq u\leq w$, applying the concavity of $u \mapsto \tilde{u}$ yields
\[
\tilde{w}=\tilde{w}+ \tilde{v} \leq  \tilde{u}\leq \tilde{w}.
\]
Thus, $\tilde{u}=\tilde{w}=w$ and $u$ is in $\Ker(\tilde{u})$.
\end{proof}

We are now prepared to present a proof of Theorem~\ref{bvinN}. This proof demonstrates the existence of boundary values for specific functions within $\E$, which has been the central focus of this section.

\begin{proof}[Proof of Theorem~\ref{bvinN}]
From Theorem \ref{thm: sol in N} and the assumption it follows that there exists $v\in \Ker$ such that $u\leq v$ and $(dd^c v)^n=(dd^c u)^n$. Lemma~\ref{thm: from sol to boundary} then ensures that $u\in \Ker(\tilde{u})$ as desired.
\end{proof}

\section{Core Open Problems in the Cegrell Classes}\label{sect: Open Problems}
  Let $\mu$ be a non-negative Radon measure defined on a bounded hyperconvex domain in $\C^n$, $n \geq 2$, and consider the Dirichlet problem for the complex Monge-Ampère equation:
\begin{equation}
\left\{
\begin{aligned}
&(dd^c u)^n = \mu, \label{eq:MA} \\
&u \in \Ker(H), \nonumber
\end{aligned}
\right.
\end{equation}
noting that the boundary values are implicitly defined within $\Ker(H)$.  In light of our recent progress, as detailed in
Theorem~\ref{thm: uniqueness in N}, this section aims to reignite the interest in the core open problems of the Cegrell classes and enhance the greater understanding in the field. We organize our discussion around the following topics:

\medskip
\begin{enumerate}
  \item Background
  \item Existence and uniqueness of solutions
  \item Boundary values
\end{enumerate}

\medskip

\noindent $(1)$ \emph{Background:} For equation~(\ref{eq:MA}) to be well-posed, it is necessary for $u$ to be in $\mathcal{E}$. A classic example of a plurisubharmonic function not in $\mathcal{E}$ is $\log|z_2|$. Furthermore, as demonstrated in \cite[Example 4.6]{AhagCegrellPham}, there exists a function $u \in \mathcal{PSH}^-(\Omega)$ satisfying $u>-\infty$ yet not belonging to $\mathcal{E}$. For $n=2$, it is established that $\mathcal{E} = \mathcal{PSH}^- \cap W^{1,2}_{\text{loc}}(\Omega)$ \cite{Blo04}. Additional characterizations of $\mathcal{E}$ can be found in \cite{Blo06, CKZ05}. A characterization of $\mathcal{F}$ is available in \cite{Benel09}, and for $\Ker$, in Section~\ref{Sec:Preliminaries}.

It is noteworthy that Kiselman \cite{Kis84} defined the complex Monge-Ampère operator using the multiplication of distributions in the sense of Colombeau. The potential for applying more contemporary distribution theory to define the complex Monge-Ampère operator remains an area yet to be explored.

\medskip

\noindent $(2)$ \emph{Existence and uniqueness of solutions:} In the context of boundary values satisfying $(dd^c H)^n \leq \mu$, it is inferred from~\cite{ACCP09} that $H=0$ can be assumed without loss of generality. The condition  that $(dd^c H)^n \leq \mu$ is trivially fulfilled if $H$ is a maximal plurisubharmonic function in $\mathcal{E}$~\cite{Blo04,BloM,Ceg09}. In classical potential theory, a positive measure $\mu$ is equal to the Laplacian of a negative subharmonic function if and only if
\begin{equation}\label{eq: Open Condition}
\int (-w)\mu < +\infty,
\end{equation}
for some negative subharmonic function $w$. Drawing inspiration from this classical potential theory, Cegrell demonstrated in~\cite{Ceg08} that if $\mu$ is null on all pluripolar sets and there exists a $w \in \mathcal{PSH}(\Omega)$, with $w < 0$, satisfying~(\ref{eq: Open Condition}), then a unique solution to~(\ref{eq:MA}) in $\Ker^a$ is guaranteed. However, Cegrell also illustrated through an example that such an equivalence does not transfer to the pluricomplex case (also see~\cite[Example 4.1]{Hiep14}). The complexity arises because a non-negative Radon measure may accumulate excessive mass near the boundary, precluding it from falling within the range of the Monge-Ampère operator. A forefront result concerning the existence of solutions can be found in~\cite[Proposition 4.3]{Hiep14} (see also Theorem~\ref{thm: sol in N}).

Let us now continue the existence of solutions when the Monge-Ampère measure can charge on pluripolar sets. To avoid the problem in previous paragraph we assume that we are instead in $\F=\{u\in\Ker:\int_{\Omega} (dd^c u)^n<+\infty\}$. It is well-known that there can be several solutions to a  Monge-Ampère equation with right hand side that can charge a pluripolar set, also with the solutions being non-comparable (see e.g. \cite{CelPol,Dem93,Lem83,Zer97}). We emphasize that, at present, there is only one known example of an atomless measure $\mu$ carried by a pluripolar set, for which a function $u \in \mathcal{F}$ exists that satisfies~(\ref{eq:MA}) (\cite[Example 4.10]{ACCP09}). To make the difficulties more transparent recall that there is function $u\in\F(\Omega)$ with $\overline{\{u=-\infty\}}=\Omega$, and $(dd^c u)^n=0$ on pluripolar sets (\cite[Theorem~5.8]{Ceg04}), but on the other hand there is a function $u\in\F(\Omega)$ with $\overline{\{u=-\infty\}}=\Omega$, and $(dd^c u)^n=\delta_0$ (\cite[Example 2.1]{AhagCegrellPham}). Even if $\mu=d\lambda\times\delta_0$ defined in the bidisc in $\mathbb{C}^2$, we are not even close to know whether there exists a $u\in\F$ with $(dd^c u)^n=\mu$. Here $d\lambda$ denotes the Lebesgue measure, and $\delta_0$ the Dirac measure, both defined in the unit disc in $\mathbb{C}$. For a few partial results see \cite{AhagCegrellPham}. The existence of the solution of the Monge-Ampère equation in $\Ker$ is far from being settled.

In Theorem~\ref{thm: uniqueness in N}, we solved the uniqueness of solutions under the assumption that the solutions are comparable in the sense of Definition~\ref{def: singular}, and we applied our result in Theorem~\ref{thm: sol in N}.

\medskip

\noindent $(3)$ \emph{Boundary values:} \emph{Given $u \in \mathcal{E}$, can we assert that $u$ also belongs to $\Ker(\tilde{u})$?} Cegrell verified this for functions $u \in \mathcal{E}$ with $\int (dd^c u)^n<+\infty$. The second-named author extended this result under the conditions that $(dd^c u)^n$ vanishes on pluripolar sets and for some function $\varphi \in \mathcal{PSH}(\Omega)$ with $\varphi < 0$, the following holds:
\[
\int_{\Omega}(-\varphi)(dd^c u)^n < +\infty
\]
(see~\cite{Cz09}). Subsequently, Ph{\d{a}}m in \cite{PhamPC} showed that the existence of boundary values does not require the vanishing of $(dd^c u)^n$ on all pluripolar sets. In Theorem~\ref{bvinN}, we presented a proof that slightly diverges from Ph{\d{a}}m's unpublished approach. However, the question of boundary values of plurisubharmonic functions in $\mathcal{E}$, as initially raised by Cegrell, remains open.

Moreover, as explored in Section~\ref{sec:RIC}, the principal properties of functions in $\Ker(H)$, such as idempotency, rooftop equality, and geodesic connectivity, stem from the properties of $H$. Selecting $H$ as $\widetilde{u}$ focuses on maximal plurisubharmonic functions in $\mathcal{E}$, which lack strong singularities inside the domain, and their boundary behavior predominantly determines the outcomes. An example illustrating strong boundary singularity is the pluricomplex Poisson kernel \cite{BPT}, discussed in \cite[Example 4.5]{Rash22}. In light of Theorem~\ref{thm: geodesic connectivity}, obtaining a method to verify whether two maximal plurisubharmonic functions share the  `same' singularity is crucial. Therefore, understanding the boundary values of these functions becomes imperative.

\medskip

This section does not aim to provide a comprehensive historical overview. For readers interested in the historical context, we recommend the following references~\cite{Bedford,Cz09,Kis00,Kol98,Krylov}.

\end{document}